\documentclass[10pt]{article}
\usepackage{amssymb}
\usepackage{amsmath}
\usepackage{amsfonts}
\usepackage{latexsym}
\usepackage[all]{xy}
\usepackage{amsmath, pb-diagram}
\usepackage{amssymb}
\usepackage{amscd}
\newtheorem{theorem}{Theorem}[section]
\newtheorem{lem}[theorem]{Lemma}
\newtheorem{thm}[theorem]{Theorem}
\newtheorem{rem}[theorem]{Remark}
\newtheorem{prop}[theorem]{Proposition}
\newtheorem{cor}[theorem]{Corollary}
\newtheorem{ex}[theorem]{Example}
\newtheorem{defn}[theorem]{Definition}

\DeclareMathAlphabet{\eusm}{OT1}{eusm}{m}{n}
\def\z*{\mbox{Z$^*$}}
\def\s*{\mbox{(S$^*$)}}

\newenvironment{proof}{\par\noindent{\itshape{Proof}.\,}}{$\hfill \Box$\par\bigskip}

\begin{document}
\title{{\bf Rings over which every module has a flat $\delta$-cover}}
\author{P\i ýnar AYDO\u GDU \\\\
Department of Mathematics,
Hacettepe University\\
 06800 Beytepe Ankara, Turkey\\
paydogdu@hacettepe.edu.tr}

\date{}
\maketitle

\begin{abstract} Let $M$ be a module. A {\em $\delta$-cover} of $M$ is an epimorphism
from a module $F$ onto $M$ with a $\delta$-small kernel. A
$\delta$-cover is said to be a {\em flat $\delta$-cover} in case
$F$ is a flat module. In the present paper, we investigate some
properties of (flat) $\delta$-covers and flat modules having a
projective $\delta$-cover. Moreover, we study rings over which
every module has a flat $\delta$-cover and call them {\em right
generalized $\delta$-perfect} rings. We also give some
characterizations of $\delta$-semiperfect and $\delta$-perfect
rings in terms of locally (finitely, quasi-, direct-) projective
$\delta$-covers and flat $\delta$-covers.
\end{abstract}

\section{Preliminaries and Notation}

Let $R$ be a ring and $\mathcal{F}$ be a class of $R$-modules. Due to Enochs and Jenda \cite{EJ}, for an
$R$-module $M$, a morphism $\varphi:C\rightarrow M$, where $C\in \mathcal{F}$, is called an $\mathcal{F}$-cover
of $M$ if the following properties are satisfied:

1) For any morphism $\psi:C'\rightarrow M$, where $C'\in \mathcal{F}$, there is a morphism
$\lambda:C'\rightarrow C$ such that $\varphi o \lambda=\psi$, and

2) if $\mu$ is an endomorphism of $C$ such that $\varphi o \mu=\varphi$, then $\mu$ is an
automorphism of $C$.

If $\mathcal{F}$ is the class of projective modules, then an $\mathcal{F}$-cover is called a
{\em projective cover}. This definition is in agreement with the usual definition of a
projective cover. If $\mathcal{F}$ is the class of flat modules, then an $\mathcal{F}$-cover is
called a {\em flat cover}. On the other hand, some authors deal with flat covers in the following sense:

Let $M$ be an $R$-module. A {\em flat cover} of $M$ is an epimorphism $f:F\rightarrow M$ with a small kernel,
where $F$ is a flat module.

In this paper, we will consider the second definition. In fact, the notion of a flat cover in this sense is a
natural generalization of a projective cover. But these two notions of flat covers do not coincide.
There are examples of modules which do not have flat covers (see \cite{AAES}) whereas all modules
have flat covers in Enochs' sense (see \cite{BE}).

 Amini, Amini, Ershad and Sharif investigate in \cite{AAES} those rings $R$ whose
right $R$-modules have flat covers, and call
them {\em right generalized perfect} ({\em right $G$-perfect}, for
short) rings.

It is well-known that projective covers play an important role in
characterizing perfect and semiperfect rings. Some authors have
also characterized these rings in terms of flat covers. Ding and
Chen show in \cite{CD} that a ring $R$ is right perfect if and
only if $R$ is semilocal and every semisimple right $R$-module has
a flat cover. In \cite{Lomp}, Lomp prove that a ring $R$ is
semiperfect if and only if $R$ is semilocal and every simple right
$R$-module has a flat cover.

Recall from \cite{zhou} that an epimorphism $f:P\rightarrow M$ with a $\delta$-small kernel is called
a {\em projective $\delta$-cover} of the module $M$ in case $P$ is projective.
As a proper generalization of perfect (resp., semiperfect) rings, $\delta$-perfect
(resp., $\delta$-semiperfect) rings are defined in \cite{zhou} as follows:
 A ring $R$ is said to be {\em $\delta$-perfect}
(resp., {\em $\delta$-semiperfect}) if every $R$-module (resp.,
simple $R$-module) has a projective $\delta$-cover.

These results motivated us to define the notion of flat
$\delta$-covers. In this paper, we deal with rings over which
(certain) right modules have flat $\delta$-covers. Firstly, in
Section 2, we investigate some basic properties of
$\delta$-covers. We prove that if a module has a flat
$\delta$-cover, then a generalized projective $\delta$-cover of
the module is a projective $\delta$-cover. It is a well-known fact
that if a flat module has a projective cover, then it is
projective. As Example~\ref{e2} shows, a flat module need not be
projective whenever it has a projective $\delta$-cover. However,
over a ring with a finitely generated right socle, a finitely
generated flat module is projective if it has a projective
$\delta$-cover. Section 3 is concerned with those rings $R$ whose
right $R$-modules have flat $\delta$-covers. We call them {\em
`right generalized $\delta$-perfect'} ({\em right $G$-$\delta$-perfect},
for short) rings and show that this notion is a proper
generalization of $\delta$-perfect rings. As Example~\ref{e1}
shows, this notion is not left-right symmetric. We prove that if
$R$ is a right $G$-$\delta$-perfect ring, then $J(R/S_r)$ is right
$T$-nilpotent. This result leads us to generalize some important
results proved in \cite{AAES}. For instance, we are able to show
that if $R$ is a right $G$-$\delta$-perfect ring, then $R$ is
right Artinian if and only if $R$ is right Noetherian. In the last
section, we give some characterizations of $\delta$-perfect and
$\delta$-semiperfect rings in terms of flat $\delta$-covers. We
also consider locally projective, finitely projective,
quasi-projective and direct-projective $\delta$-covers in order to
give some necessary and sufficient conditions for a ring to be
$\delta$-perfect or $\delta$-semiperfect.

Throughout this paper, $R$ denotes an associative ring with
identity and modules are unitary right $R$-modules. For a module
$M$, $Soc(M)$ is the socle and $Rad(M)$ is the Jacobson radical of
$M$. $S_r$ and $J(R)$ will stand for the right socle and the
Jacobson radical of a ring $R$, respectively. We will denote a
direct summand (resp., small submodule) of a module $M$ by
 $K\leq^{\oplus} M$ (resp., $K\ll M$).

 As a generalization of small submodules, in \cite{zhou}, Zhou introduce
$\delta$-small submodules as follows:

A submodule $N$ of a module $M$ is said to be {\em $\delta$-small} if $N+K\neq M$ for any proper submodule
$K$ of $M$ with $M/K$ singular, and it is denoted by $N\ll_{\delta} M$.
By this definition, every small or nonsingular semisimple submodule of $M$ is $\delta$-small in $M$.

The below lemma, which is appeared in \cite{zhou}, gives a necessary and sufficient condition for a submodule $N$ of
$M$ to be $\delta$-small in $M$ and we will use it throughout the paper.

\begin{lem}\label{lz} \cite[Lemma 1.2]{zhou} The following are equivalent:

$(1)$ $N\ll_{\delta} M$

$(2)$ If $X+N=M$, then $M=X\oplus Y$ for a projective semisimple submodule $Y$ with $Y\subseteq N$.
\end{lem}

According to \cite[Lemma 1.5]{zhou},  the submodule $\delta(M)=\sum \{L\subseteq M| L\ll_{\delta} M
\}$ which is also equal to the
intersection of all essential maximal submodules of $M$ whenever
$M$ is projective (see \cite[Lemma 1.9]{zhou}). We will use the
notation $\delta_r$ to indicate the intersection of all essential
maximal right ideals of $R$. Note from \cite[Corollary 1.7]{zhou} that $J(R/S_r)=\delta_r/S_r$.

\section{Flat $\delta$-covers}

\begin{defn} {\rm An epimorphism $f:P\rightarrow M$ is called a {\em $\delta$-cover of $M$} in case $Ker(f)\ll_{\delta} P$.}
\end{defn}

We start with some basic properties of $\delta$-covers. The proofs of the following
three lemmas are straightforward, so we omit them.

\begin{lem}\label{l1} If $f:P\rightarrow M$ and $g:M\rightarrow N$ are $\delta$-covers, then $gf:P\rightarrow N$ is a $\delta$-cover.
\end{lem}

\begin{lem}\label{b1} If each $f_i:P_i\rightarrow M_i$ is a $\delta$-cover for $i=1,\ldots,n$, then
$\oplus_{i=1}^{n}f_i:\oplus_{i=1}^{n} P_i\rightarrow \oplus_{i=1}^{n} M_i$ is a $\delta$-cover.
\end{lem}

\begin{lem}\label{b3} If $N\leq^{\oplus} M$ and $A\ll_{\delta} M$, then $A\cap N\ll_{\delta} N$.
\end{lem}

\begin{lem}\label{b4} Let $K$ be a submodule of a projective module $F$.
If $F/K$ has a $\delta$-cover, then it has a $\delta$-cover of the form
$f:F/L\rightarrow F/K$ with $Ker(f)=K/L$, where $L\subseteq K$.
\end{lem}

\begin{proof} Let $f:P\rightarrow F/K$ be a $\delta$-cover
of $F/K$ and $\pi:F\rightarrow F/K$ be the natural epimorphism.
Since $F$ is projective, there exists a homomorphism
$\lambda:F\rightarrow P$ such that the below diagram commutes.

\[\begin{diagram}
\node[2]{F}\arrow{sw,t,..}{\lambda}\arrow{s,r}\pi \\
\node{P}\arrow{e,t}{f}\node{F/K}\arrow{e}\node{0}
\end{diagram}\]
Then $P=Ker(f)+Im(\lambda)$. It follows from Lemma~\ref{lz}
that $P=Y\oplus Im(\lambda)$ for a semisimple
submodule $Y$ with $Y\subseteq Ker(f)$. Also, by
Lemma~\ref{b3}, $Ker(f|_{Im(\lambda)})=Im(\lambda)\cap
Ker(f)\ll_{\delta} Im(\lambda)$. So $f|_{Im(\lambda)}:
Im(\lambda)\rightarrow F/K$ is also a $\delta$-cover of $F/K$. But
$F/Ker(\lambda)\cong Im(\lambda)$ and since $f \lambda=\pi$,
$Ker(\lambda)\subseteq Ker(\pi)=K$. Now consider the isomorphism $\lambda':F/Ker(\lambda)\rightarrow Im(\lambda)$
and let $\phi:=f|_{Im(\lambda)} \lambda': F/Ker(\lambda)\rightarrow F/K$. Then $Ker(\phi)=K/Ker(\lambda)$ and by
Lemma~\ref{l1}, $Ker(\phi)\ll_{\delta} F/Ker(\lambda)$.
\end{proof}

Since any finitely generated (resp., cyclic) module is an epimorphic image of a finitely generated
(resp., cyclic) free module,  we obtain the following result by the proof of Lemma~\ref{b4}.

\begin{lem}\label{b5} If $f:P\rightarrow M$ is a $\delta$-cover of a finitely generated (cyclic) module $M$,
 then there exists a finitely generated (cyclic) direct summand $P'$ of $P$ such that
$f|_{P'}$ is a $\delta$-cover of $M$.
\end{lem}

\begin{defn}
{\rm A $\delta$-cover  $f:P\rightarrow M$ is called a {\em flat
$\delta$-cover} of $M$ in case $P$ is a flat module.}
\end{defn}

It is clear that if a module has a projective $\delta$-cover, then
it also has a flat $\delta$-cover. By Example~\ref{e1} below, the
converse does not hold in general. Now we will investigate under
which condition a module $M$ has a projective $\delta$-cover
whenever it has a flat $\delta$-cover. But we need some results in
order to prove one of the main result of this section.

Locally projective modules are introduced by Zimmermann-Huisgen
(\cite{Zim}) and we know from \cite[Proposition 6]{Azumaya} that an
$R$-module $M$ is {\em locally projective} if and only if for any
$x\in M$ there exist a finite number of homomorphisms
$f_i:M\rightarrow R$ ($i=1,\ldots,n$) and elements $y_i\in M$
($i=1,\ldots,n$) such that $y_1f_1(x)+\cdots+y_nf_n(x)=x$. It is
well-known that the following implications hold for a module:

\begin{center}
projective $\Rightarrow$ locally projective $\Rightarrow$ flat.
\end{center}

\begin{prop}\label{p1} If $M$ is a locally projective module, then $M \delta_r=\delta(M)$.
\end{prop}

\begin{proof} By \cite[Lemma 1.5(2)]{zhou}, the inclusion $M \delta_r\subseteq \delta(M)$ always holds.
 For the reverse inclusion let $x\in \delta(M)$. Then by hypothesis,
  there exist a finite number of homomorphisms $f_i:M\rightarrow R$
and elements $y_i\in M$ ($i=1,\ldots,n$) such that $y_1f_1(x)+\cdots+y_nf_n(x)=x$.
It follows from \cite[Lemma 1.5(2)]{zhou} that
$f_i(\delta(M))\subseteq \delta_r$ for each $i$ and so $f_i(x)\in \delta_r$ for each $i$.
 Hence, we obtain that $x\in M \delta_r$.
\end{proof}

\begin{defn}  {\rm An epimorphism $f:P\rightarrow M$ is called a {\em generalized (locally) projective
$\delta$-cover of $M$}
 in case $Ker(f)\subseteq\delta(P)$ and $P$ is (locally) projective.}
\end{defn}

For a homomorphism $f:P\rightarrow M$, the inclusion
$f(\delta(P))\subseteq \delta(M)$ always holds by \cite[Lemma
1.5(2)]{zhou}. It can be observed that the equality holds whenever
$f:P\rightarrow M$ is an epimorphism and $Ker(f)\subseteq
\delta(P)$. By this fact, we obtain the following result.

\begin{cor} If a module $M$ has a generalized locally projective $\delta$-cover, then $M \delta_r=\delta(M)$.
\end{cor}

\begin{proof} Let $f:P\rightarrow M$ be a generalized locally projective $\delta$-cover of $M$.
Then $\delta(M)=f(\delta(P))=f(P\delta_r)=f(P)\delta_r=M\delta_r$.
\end{proof}

\begin{prop}\label{p2} If $M$ is a locally projective module, then $MS_r=Soc(M)$.
\end{prop}

\begin{proof} It follows from a proof similar to that of Proposition~\ref{p1}.
\end{proof}

\begin{rem}\label{r1} {\rm

1) Note that $[\delta(M)+Soc(M)]/Soc(M)\subseteq Rad(M/Soc(M))$
for any module $M$: Consider $\overline{m}=m+Soc(M)\in
[\delta(M)+Soc(M)]/Soc(M)$, where $m\in \delta(M)$. Suppose that
$\overline{m}\notin Rad(M/Soc(M))$. Then there exists a maximal
submodule of $M$ with $Soc(M)\subseteq L$ and $m\notin L$. So
$M=L+mR$. Since $mR\ll_{\delta} M$, $M=L\oplus Y$ for a projective
semisimple submodule $Y$ of $mR$. But $Soc(M)\subseteq L$ implies
that $Y=0$.
 It follows that
$M=L$, a contradiction.

2) It is easy to observe that if $P$ is a locally projective
$R$-module, then $P/PI$ is a locally projective $R/I$-module for
any ideal $I$ of $R$.

3) We know from \cite[Proposition 2.2]{Zim} that a locally
projective module with $Rad(M)=M$ is zero.

4) Recall from \cite[Proposition 10]{Azumaya} that a countably
generated locally projective module is projective.}
\end{rem}

\begin{prop}\label{p3} Let $M$ be a locally projective module with $\delta(M)=M$.
Then $M$ is a projective semisimple module.
\end{prop}

\begin{proof} Since $M=\delta(M)$, we get that
$Rad(M/Soc(M))=M/Soc(M)$ by Remark~\ref{r1}(1). Also,
Remark~\ref{r1}(2) together with Proposition~\ref{p2} implies that
$M/Soc(M)$ is a locally projective $R/S_r$-module. It follows from
Remark~\ref{r1}(3) that $M=Soc(M)$. Moreover, $M$ is projective
because a simple locally projective module is projective by
Remark~\ref{r1}(4).
\end{proof}

Recall from \cite{Lam} that a short exact sequence of right
$R$-modules $0\rightarrow A\stackrel{\varphi} \rightarrow B
\rightarrow C\rightarrow 0$ is {\em pure} if it remains exact
after being tensored with any left $R$-module. If this is the
case, then $\varphi(A)$ is called {\em a pure submodule} of $B$.
It is known that direct summands are pure submodules. Due to
\cite[Theorem 4]{pd}, if $N$ is a finitely generated pure
submodule of a projective module $P$, then it is a direct summand
of $P$. Let $A\subseteq B\subseteq D$ be right $R$-modules. If $A$
is pure in $B$ and $B$ is pure in $D$, then $A$ is pure in $D$
(see \cite[Examples 4.84(e)]{Lam}). Also, it follows from
\cite[Theorem 4.85]{Lam} that if $M/N$ is a flat $R$-module, then
$N$ is a pure submodule of $M$, and the converse holds if $M$ is
flat by \cite[Corollary 4.86(1)]{Lam}. We know from
\cite[Corollary 4.92]{Lam} that if $N$ is a pure submodule of $M$,
then $NI=N\cap MI$ for each left ideal $I$ of the ring $R$. If $M$
is a projective module, then the converse holds by \cite[Exercise
41, pg. 163]{Lam}. In addition, pure submodule of a locally
projective module is locally projective by \cite[Proposition
7]{Azumaya}.

Now we are ready to prove the following
result as promised.
\begin{thm}\label{t2} Suppose that a module $M$ has a flat
$\delta$-cover.
A generalized projective $\delta$-cover of $M$ is a
projective $\delta$-cover of $M$.
\end{thm}

\begin{proof} Let $f:X\rightarrow M$ be a flat $\delta$-cover and
$g:P\rightarrow M$ a generalized projective $\delta$-cover of $M$.
$P$ being projective implies that there exists a homomorphism
$h:P\rightarrow X$ such that $fh=g$. Then $X=Ker(f)+Im(h)$. Since
$Ker(f)\ll_{\delta} X$, $X=T\oplus Im(h)$ for a projective
semisimple submodule $T$ with $T\subseteq Ker(f)$ by
Lemma~\ref{lz}. As $Im(h)\leq^{\oplus} X$, $Im(h)$ is also a flat
module. Hence, $P/Ker(h)$ is flat and so $Ker(h)$ is a pure
submodule of $P$. Moreover, $Ker(h)$ is locally projective. On the
other hand, $Ker(h)\subseteq Ker(g)\subseteq \delta(P)$. So due to
\cite[Corollary 4.92]{Lam} the purity of $Ker(h)$ implies that
$Ker(h) \delta_r=Ker(h)\cap P\delta_r=Ker(h)\cap
\delta(P)=Ker(h)$. But the fact that $Ker(h)$ is locally
projective together with Proposition~\ref{p1} implies that
$\delta(Ker(h))=Ker(h)$. Hence, $Ker(h)$ is projective semisimple
by Proposition~\ref{p3}, which means that $Ker(h)\ll_{\delta} P$.
So $h:P\rightarrow Im(h)$ is a projective $\delta$-cover of
$Im(h)$. We can also observe that $f|_{Im(h)}h=g$, where
$f|_{Im(h)}$ is also a flat $\delta$-cover of $M$. So by
Lemma~\ref{l1}, we get that $Ker(f|_{Im(h)}h)=Ker(g)\ll_{\delta}
P$, as desired.
\end{proof}

Using the idea of the proof of \cite[Theorem 10.5.3]{Kasch} we
obtain the following theorem.
Note that this result can also be used to prove Theorem~\ref{t2}. Indeed, by Proposition~\ref{t4},
the submodule $Ker(h)$ in the proof of Theorem~\ref{t2} is projective semisimple.

\begin{prop}\label{t4} Suppose that $P$ is a projective module, $U\subseteq
\delta(P)$ and $P/U$ is flat. Then $U$ is projective semisimple. In this case,
every finitely generated submodule of $U$ is a direct summand of $P$.
\end{prop}

\begin{proof} Firstly we will prove the theorem in case $P=F$ is a
free module. Let $\{x_i| i\in I\}$ be a basis of $F$. Take $u\in
U$ and let $u=\sum_{i=1}^n x_i a_i$, where $a_i\in R$. Consider
the finitely generated left ideal $A=\sum_{i=1}^n Ra_i$ of $R$.
Since $F/U$ is flat, $U$ is pure in $F$. So, by \cite[Corollary
4.92]{Lam}, we have that $U\cap FA=UA$. Then $u\in U\cap FA=UA$.
Hence, there exist $u_j\in U$ and $b_j\in A$ such that
$u=\underset{j} \sum u_jb_j$. Since $U\subseteq
\delta(F)=F\delta_r$, we have that $u_j=\underset{k} \sum x_k
c_{jk}$, where $c_{jk}\in \delta_r$. So $u=\underset{i} \sum
x_ia_i=\underset{j} \sum \underset{k} \sum x_k c_{jk} b_j$ gives
us that $a_i=\underset{j} \sum c_{ji} b_j$ implying that
$A=\delta_r A$. Now we can observe that
$\frac{A+S_r}{S_r}=\frac{\delta_r A+S_r}{S_r}=\frac{\delta_r}{S_r}
\frac{A+S_r}{S_r}=J(\frac{R}{S_r}) \frac{A+S_r}{S_r}$. By
Nakayama's Lemma, $A\subseteq S_r$. Then $u\in U\cap FS_r=U\cap
Soc(F)=Soc(U)$. Hence, $U=Soc(U)$. Since $U$ is a pure submodule
of a projective module, it is locally projective. But
semisimplicity of $U$ implies that $U$ is projective.

Now let $P$ be a projective module and $U$ be a submodule of $P$
such that $U\subseteq \delta(P)$ and $P/U$ is flat. Since $P$ is a
direct summand of a free module $F$, $F=P\oplus P_1$ for some
$P_1\leq F$. Consider the natural epimorphism $\pi:F\rightarrow
F/U$. We have that $F/U=\pi(F)=\pi(P)+\pi(P_1)$. Since $U\subseteq
P$, this sum is direct. Also, $\pi(P)=P/U$ and
$\pi(P_1)=(P_1+U)/U\cong P_1$. Hence, $F/U\cong P/U \oplus P_1$ is
flat. By hypothesis, $U\subseteq \delta(P)\subseteq \delta(F)$.
From the proof above $U$ is projective semisimple.

In this case, every submodule of $U$ is a direct summand of $U$
and so a pure submodule of $P$. Because $U$ is a pure submodule of
$P$, every finitely generated submodule of $U$ is a direct summand
of $P$ by \cite[Theorem 4]{pd}.
\end{proof}

By Proposition~\ref{t4}, we obtain the following result which will turn out to be a useful tool
in characterizing $\delta$-semiperfect rings in Section 4.

\begin{prop}\label{p4} If a flat module $F$ has a projective $\delta$-cover,
then so does every finitely generated pure submodule of $F$.
\end{prop}

\begin{proof} Let $f:P\rightarrow F$ be a projective $\delta$-cover of $F$.
Then $Ker(f)$ is projective semisimple by Proposition~\ref{t4}.
Consider a finitely generated pure submodule $L=\sum_{i=1}^n x_iR$
of $F$. Since $f$ is epic, there exists $p_i\in P$ such that
$f(p_i)=x_i$ for each $i=1,\ldots,n$. So $\sum_{i=1}^n
p_iR\subseteq T=f^{-1}(L)$. To show that $T=Ker(f)+\sum_{i=1}^n
p_iR$, let $t\in T$. Then $f(t)=\sum_{i=1}^n x_ir_i=f(\sum_{i=1}^n
p_ir_i)$ ($r_i\in R$) which gives that $t-\sum_{i=1}^n p_ir_i\in
Ker(f)$. Hence, we get the desired equality. As $Ker(f)$ is
projective semisimple $Ker(f)\ll_{\delta} T$ so that
$T=\sum_{i=1}^n p_iR\oplus Y$, where $Y$ is a projective
semisimple submodule of $Ker(f)$. On the other hand, because $F$
is flat and $L$ is pure $P/T\cong F/L$ is flat which means that
$T$ is a pure submodule of $P$. It follows that $\sum_{i=1}^n
p_iR$ is also a pure submodule of $P$ and so it is projective by
\cite[Theorem 4]{pd}. So, $T$ is projective. Thus,
$f|_T:T\rightarrow L$ is a projective $\delta$-cover of $L$.
\end{proof}
It is known that if a flat module has a projective cover, then it is projective. But, as the
following example shows,
this is not the case for a flat module which has a projective $\delta$-cover even if the flat module is cyclic.

\begin{ex}\label{e2}{\rm \cite[Example 4.1]{zhou} Let $Q=\prod_{i=1}^{\infty} F_i$, where
each $F_i=\mathbb{Z}_2$. Let $R$ be the subring of $Q$ generated
by $S=\oplus_{i=1}^{\infty} F_i$ and $1_Q$. Consider the singular
simple $R$-module $R/S$. Since $R$ is a (von Neumann) regular
ring, $R/S$ is a flat $R$-module. Zhou shows that $R$ is a
$\delta$-semiperfect ring so that $R/S$ has a projective
$\delta$-cover. If $R/S$ was projective, then $R$ would be
semisimple, which is a contradiction.}
\end{ex}

On the other hand, we obtain the following result.

\begin{prop}\label{fp} Let $R$ be a ring with a finitely generated right socle $S_r$. If $F$ is a finitely generated
flat module with a projective $\delta$-cover, then it is projective.
\end{prop}

\begin{proof} Let $f:P\rightarrow F$ be a projective $\delta$-cover of a finitely generated
flat module $F$. Then, by Lemma~\ref{b5}, we can assume that $P$
is also finitely generated. By Theorem~\ref{t4}, $Ker(f)$ is
projective semisimple so that $Ker(f)\subseteq Soc(P)=PS_r$.
$Soc(P)$ being finitely generated implies that $Ker(f)$ is
finitely generated. But $Ker(f)$ is a pure submodule of $P$.
Hence, $Ker(f)\leq^{\oplus} P$ by \cite[Theorem 4]{pd}. Thus, $F$
is projective.
\end{proof}

As Example~\ref{e2} shows, the condition that `$S_r$ is finitely generated'
is not superfluous in Proposition~\ref{fp}.

\section{Generalized $\delta$-perfect rings}

\begin{defn} {\rm A ring $R$ is said to be {\em right generalized $\delta$-perfect}
 ({\em right $G$-$\delta$-perfect}, for short) if every right $R$-module has a flat $\delta$-cover.
  Left $G$-$\delta$-perfect rings are defined similarly. We call $R$ a {\em
  $G$-$\delta$-perfect}
ring in case it is both right and left $G$-$\delta$-perfect.}
\end{defn}

We start this section with some examples.

\begin{ex}{\rm Trivially, every flat module has a flat $\delta$-cover. Hence, every regular ring is $G$-$\delta$-perfect.}
\end{ex}

\begin{ex} {\rm A right $\delta$-perfect ring is a right
$G$-$\delta$-perfect ring. The converse need not be true as Example~\ref{e1} shows.}
\end{ex}

\begin{ex} {\rm $\mathbb{Z}$ is not a $G$-$\delta$-perfect ring.}
\end{ex}

\begin{proof} Let $n\geq 2$ and consider the $\mathbb{Z}$-module $M=\mathbb{Z}/n\mathbb{Z}$.
Assume that $f:F\rightarrow M$ is a flat
$\delta$-cover of $M$. From the proof of Lemma~\ref{b4}
 we get that M has a flat $\delta$-cover of the form $\mathbb{Z}/K$ which is
isomorphic to $F$ because projective semisimple
$\mathbb{Z}$-modules are zero. So $\mathbb{Z}/K$ is a cyclic flat
$\mathbb{Z}$-module. But it is projective since $\mathbb{Z}$ is
Noetherian.
 Then $K\leq^{\oplus} \mathbb{Z}$. As $K\neq \mathbb{Z}$ we obtain that $K=0$. So $F\cong \mathbb{Z}$.
Let $g:F\rightarrow \mathbb{Z}$ be the isomorphism. Since
$Ker(f)\ll_{\delta} F$, $g(Ker(f))\ll_{\delta} \mathbb{Z}$ by
\cite[Lemma 1.3(2)]{zhou}. Since $\delta(\mathbb{Z})=0$ and $g$ is
an isomorphism, we have that $Ker(f)=0$. So $f$ is an isomorphism
which means that $M\cong \mathbb{Z}$.
 But this is a contradiction. Thus, M does not have a flat $\delta$-cover.
\end{proof}

\begin{ex} \label{e3}{\rm Let $\mathbb{Q}$ be the set of rational numbers.
Since $\mathbb{Q}$ is a flat $\mathbb{Z}$-module and
$\mathbb{Z}\ll_{\delta} \mathbb{Q}$, the natural epimorphism
$\pi:\mathbb{Q}\rightarrow \mathbb{Q}/\mathbb{Z}$ is a flat
$\delta$-cover of $\mathbb{Q}/\mathbb{Z}$. But it can be shown by
a proof similar to that of \cite[Example 2.1(d)]{AAES} that its
direct summand $\mathbb{Z}_{p^{\infty}}$ (the Prufer $p$-group)
does not have a flat $\delta$-cover.}
\end{ex}

Example~\ref{e3} shows that a submodule of a module which has a flat $\delta$-cover need
not have a flat $\delta$-cover. However, we have the following result.

\begin{prop} Let $R$ be a ring such that $\delta(M)=M\delta_r\ll_{\delta} M$ for any flat module $M$. Assume that
$L/K$ is a flat module, where $K\subseteq L$. If $L$ has a flat $\delta$-cover, then so does $K$.
\end{prop}

\begin{proof} Assume that $f:F\rightarrow L$ is a flat $\delta$-cover of $L$ and $K\subseteq L$.
Let $P=f^{-1}(K)$. Then $F/P\cong L/K$ is flat and so $P$ is flat by \cite[Corollary 4.86]{Lam}.
Also, we have that $Ker(f)\subseteq F\delta_r\cap P=P\delta_r$ since $P$ is a pure submodule of $F$. By assumption,
$P\delta_r\ll_{\delta} P$ and so $Ker(f)\ll_{\delta} P$. Hence, we obtain that $f:P\rightarrow K$ is a flat
$\delta$-cover of $K$.
\end{proof}

Now we consider some basic properties of right $G$-$\delta$-perfect rings.

\begin{prop}\label{ph} $1)$ Being a right $G$-$\delta$-perfect ring is a Morita invariant.

$2)$ The class of right $G$-$\delta$-perfect rings is closed under taking quotient rings.

$3)$ The class of right $G$-$\delta$-perfect rings is closed under finite direct product of rings.
\end{prop}

\begin{proof} $1)$ Similar to \cite[Proposition 5.14]{AF} we can easily observe that
$K\ll_\delta M$ if and only if for every module $N$ and for every
homomorphism $h:N\rightarrow M$ $Im(h)+K=M$ with $M/Im(h)$
singular implies that $Im(h)=M$. As a consequence of this result
(similar to \cite[Corollary 5.15]{AF}) we get that an epimorphism
$g:M\rightarrow N$ has a $\delta$-small kernel if and only if for
all homomorphism $h$ with $M/Im(h)$ singular if $gh$ is epic, then
$h$ is epic. Combining this fact with \cite[Exercise 18.2,
pg.501]{Lam} and with \cite[Lemma 21.3]{AF} we obtain that the
property that `having a $\delta$-cover' is preserved under a
category equivalence. Hence, by \cite[Exercise 22.12, pg.268]{AF},
we get the desired result.

$2)$ Let $I$ be an ideal of a right $G$-$\delta$-perfect ring $R$.
Consider a right $R/I$-module $M$. By hypothesis, $M$ has a flat
$\delta$-cover $f:F\rightarrow M$ as an $R$-module. Since
$f(FI)=0$, we can consider the epimorphism
$\overline{f}:F/FI\rightarrow M$ which is induced by $f$.
Moreover, this epimorphism is a flat $\delta$-cover of the
$R/I$-module $M$ because $F/FI\cong F\otimes_R R/I$ is a flat
$R/I$-module and $Ker(\overline{f})=Ker(f)/FI\ll_{\delta} F/FI$ by
\cite[Lemma 1.3(2)]{zhou}.

$3)$ It is enough to prove that $R=R_1 \times R_2$ is a right
$G$-$\delta$-perfect ring whenever the rings $R_1$ and $R_2$ are
right $G$-$\delta$-perfect. Let $M$ be a right $R$-module. If we
consider the central idempotent $e=(1,0)\in R$, then $M=Me\oplus
M(1-e)$. Since $Me$ has an $R_1$-module structure, it has a flat
$\delta$-cover $f:F\rightarrow Me$ as an $R_1$-module. $F$ is also
a flat $R$-module by \cite[Theorem 4.24]{Lam}. Let $K=Ker(f)$. To
show that $K\ll_{\delta} F$ as an $R$-module, let $F=K+T$, where
$F/T$ is singular. Then $Fe=Ke+Te$ as an $R_1$-module and $Fe/Te$
is a singular $R_1$-module. But $K\ll_{\delta} F$ as an
$R_1$-module so that $Ke\ll_{\delta} Fe$. Hence, $Fe=Te$ which
means that $F=T$. Similarly, it can be shown that $M(1-e)$ has a
flat $\delta$-cover as an $R$-module. Thus, $M$ has a flat
$\delta$-cover by Lemma~\ref{b1}.
\end{proof}

\begin{ex}\label{e1} There exists a right $G$-$\delta$-perfect ring that is not
right $\delta$-perfect.
\end{ex}

\begin{proof} Consider a non-semisimple regular ring $R$ with
$\delta_r=0$ and a right $\delta$-perfect ring $S$ that is not
regular (For examples of such rings see \cite[Examples 4.2 and
4.3]{zhou}). Then the ring $R\times S$ is right
$G$-$\delta$-perfect by Proposition~\ref{ph}(3), but it is not
right $\delta$-perfect since $(R\times S)/\delta(R\times S)\cong
R\times S/\delta(S)$ is not semisimple. Note also that $R\times S$
is not regular.
\end{proof}

Recall that a subset $S$ of a ring $R$ is said to be {\em right $T$-nilpotent} in case
for every sequence $a_1,a_2,\ldots$ in $S$ there is an integer $n\geq 1$ such that
$a_n\ldots a_2a_1=0$. The following theorem describes the right $T$-nilpotency of
$J(R/S_r)$.

\begin{thm}\label{t1} The following statements are equivalent:

1) $J(R/S_r)$ is right $T$-nilpotent.

2) $\delta(M)\ll_{\delta} M$ for every (non-semisimple) projective
module $M$.

3) $\delta(F)\ll_{\delta} F$ for every countably generated
(non-semisimple) free module $F$.
\end{thm}

\begin{proof} $(1)\Rightarrow (2)$ Let $M=\delta(M)+K$ with $M/K$
singular for a proper submodule $K$ of a projective module $M$. Since $M$
is projective, $K\leq_e M$ which implies that
$Soc(M)=MS_r\subseteq K$. So $\frac{M}{K}$ is a nonzero right
$\frac{R}{S_r}$-module. But by \cite[Lemma 28.3(b)]{AF},
$\frac{M}{K} J(\frac{R}{S_r})\neq \frac{M}{K}$ which means that
$\frac{M}{K} \frac{\delta_r}{S_r}\neq \frac{M}{K}$. On the other
hand, $\frac{M}{K} \frac{\delta_r}{S_r}=\frac{(M\delta_r+K)}{K}
\frac{R}{S_r}=\frac{M}{K}$, a contradiction. Consequently, $M=K$.

$(2)\Rightarrow (3)$ It is obvious.

$(3)\Rightarrow (1)$ It follows from a proof similar to that of
$(4)\Rightarrow (1)$ of Theorem 3.7 in \cite{zhou}. We give the
proof for completeness. Let $F\cong R^{(\aleph_0)}$ be the free
module with a basis $\{x_1, x_2, \ldots,\}$. If we let $a_1,a_2,
\ldots$ be a sequence in $\delta_r$ and $G=\sum_{i=1}^{\infty}
(x_i-x_{i+1}a_i)R$, then $F=G+\delta(F)$. By hypothesis,
$\delta(F)\ll_{\delta} F$. So $F=G\oplus Y$ for a semisimple
submodule $Y$ of $\delta(F)$ by Lemma~\ref{lz}. It follows from
\cite[Lemma 28.2]{AF} that there exists a number $n$ such that
$Ra_{n+1}a_n\cdots a_1=Ra_n\cdots a_1$. So $a_n\cdots
a_1=ra_{n+1}a_n\cdots a_1$ which implies that
$(1-ra_{n+1})a_n\cdots a_1=0$. Since $J(R/S_r)= \delta_r/S_r$ (see
\cite[Corollary 1.7]{zhou}), $a_n \cdots a_1\in S_r$ which means
that $J(R/S_r)$ is right $T$-nilpotent.
\end{proof}

In \cite{AAES}, it is shown that if $R$ is a right $G$-perfect
ring, then $J(R)$ is right $T$-nilpotent. But it is evident from
\cite[Example 4.3]{zhou} that $\delta_r$ need not be right
$T$-nilpotent whenever $R$ is a right $G$-$\delta$-perfect ring.
However, considering the characterization of $\delta$-perfect
rings (see \cite[Theorem 3.8]{zhou}), it is natural to expect the
following result.

\begin{thm}\label{t3} If $R$ is a right $G$-$\delta$-perfect ring, then
$J(R/S_r)$ is right $T$-nilpotent. In particular, idempotents lift
modulo $\delta_r$.
\end{thm}

\begin{proof} By Theorem~\ref{t1}, it is enough to show that $\delta(F)\ll_{\delta}
F$ for a countably generated free $R$-module $F$. By assumption,
$F/\delta(F)$ has a flat $\delta$-cover. Also, the natural
epimorphism $\pi:F\rightarrow F/\delta(F)$ is a generalized
projective $\delta$-cover of $F/\delta(F)$. It follows from
Theorem~\ref{t2} that $\pi$ is a projective $\delta$-cover. Hence,
$Ker(\pi)=\delta(F)\ll_{\delta} F$. In particular, idempotents of
the ring $R/S_r$ lift modulo $J(R/S_r)$. By \cite[Lemma 1.3]{YZ},
idempotents of $R$ lift modulo $\delta_r$.
\end{proof}

\begin{rem}{\rm Note that alternatively Theorem~\ref{t3} can also be proved with the
help of Proposition~\ref{t4}. We can consider
$(R/\delta_r)^{\mathbb{(N)}}$ and its flat $\delta$-cover. Then
apply the proof of Lemma~\ref{b4} considering $R^{\mathbb{(N)}}$.
The rest of the proof follows from Proposition~\ref{t4} and
Lemma~\ref{l1}.}
\end{rem}

The next example shows that the notion of $G$-$\delta$-perfect rings is not left-right symmetric.
\begin{ex} There exists a right $G$-$\delta$-perfect ring that is not
left $G$-$\delta$-perfect.
\end{ex}

\begin{proof} Let $R$ be the ring of all countably infinite
square upper triangular matrices over a field $F$ that are
constant on the main diagonal and have only finitely many nonzero
entries off the main diagonal. It is shown in (\cite[Example
B.46]{QF}) that $J(R)$ is not left $T$-nilpotent. So $J(R/S_r)$ is
not left $T$-nilpotent. Hence, $R$ is not
left $G$-$\delta$-perfect by Theorem~\ref{t3}. On the other hand,
$R$ is right $G$-$\delta$-perfect since $R$ is right perfect.
\end{proof}

According to \cite[Theorem 3.5]{zhou}, a ring $R$ is $\delta$-semiregular
if and only if $R/\delta_r$ is regular and idempotents lift modulo $\delta_r$.
B\"{u}y\"{u}ka\c{s}\i k and Lomp prove in \cite{BL} that a $\delta$-semiperfect
ring with a finitely generated right socle is semiperfect. This fact together with
Theorem~\ref{t3} enables us to prove the following result which generalizes
Proposition 2.4 in \cite{AAES}.

\begin{prop} Let $R$ be a right $G$-$\delta$-perfect ring. Then
$R$ is right Noetherian if and only if $R$ is right Artinian.
\end{prop}

\begin{proof} The necessity is obvious. For the sufficiency, let
$M$ be a simple $R$-module. Then, by Lemma~\ref{b5}, $M$ has a
flat $\delta$-cover $f:F\rightarrow M$ such that $F$ is cyclic.
Since finitely generated flat modules are projective over a
Noetherian ring, $F$ is projective. Hence, every simple $R$-module
has a projective $\delta$-cover which means that $R$ is
$\delta$-semiperfect. By \cite[Remark 4.4]{BL}, $R$ is
semiperfect. It follows that $R/S_r$ is semiperfect. Since
$J(R/S_r)$ is nil by Theorem~\ref{t3}, $R/S_r$ is right Noetherian
semiprimary ring. It follows from Hopkin's Theorem that $R/S_r$ is
an Artinian ring and so an Artinian $R$-module. Since $S_r$ is
Artinian, $R$ is right Artinian.
\end{proof}

\begin{thm}\label{r} Let $R$ be a ring such that every cyclic flat right
$R$-module is projective. If $R$ is right $G$-$\delta$-perfect, then
$R/\delta_r$ is regular.
\end{thm}

\begin{proof} By Proposition~\ref{ph}(2), it is enough prove that $R$
is regular whenever $\delta_r=0$. Assume that $R$ is not regular.
Then there exists a cyclic right $R$-module $M$ that is not flat
by \cite[Theorem 4.21]{Lam}. But $M$ has a flat $\delta$-cover
$f:F\rightarrow M$ and since $M$ is cyclic we can assume that $F$
is cyclic by Lemma~\ref{b5}. Then $F$ is projective by hypothesis.
Therefore, we get that $Ker(f)\subseteq \delta(F)=F\delta_r=0$.
Thus, $F\cong M$ is projective, which is a contradiction.
\end{proof}

Recall from \cite[pg.297 and 321]{Lam} that a ring $R$ is called
{\em strongly ($\pi$-)regular} if, for any $a\in R$, there exists
$x\in R$ (and a positive integer $n$) such that $a=a^2x$
($a^n=a^{n+1}x$). Recall also that a ring $R$ is said to be {\em right} (resp., {\em left})
{\em duo} in case every right (resp., left) ideal of $R$ is a two-sided ideal.
It is known that a strongly regular ring is
right and left duo. By the next theorem, we can conclude that a
right duo and a right $G$-$\delta$-perfect ring with $J=0$ is
strongly regular. Note also that the next theorem is a generalization of
\cite[Theorem 2.7]{AAES} since strongly regular rings are regular.

\begin{thm}\label{td} If $R$ is right duo and right $G$-$\delta$-perfect, then
$R/J$ is strongly regular.
\end{thm}

\begin{proof} By Proposition~\ref{ph}(2), we can assume that $J=0$
without loss of generality. Let $x$ be a nonzero element of $R$.
By Lemma~\ref{b4}, $R/xR$ has a flat $\delta$-cover of the form
$f:R/I\rightarrow R/xR$, where $I\subseteq xR$ and $Ker(f)=xR/I$.
Hence, $xR/I\subseteq \delta(R/I)$. $R/I$ being right
$G$-$\delta$-perfect implies that $\delta(R/I)/Soc(R/I)$ is nil so
that there exists a positive integer n such that
$\overline{x}^n=x^n+I\in Soc(R/I)$. Since $\overline{x}^nR$ is
semisimple and finitely generated, it is an Artinian $R$-module.
It follows that there exists a positive integer $k\geq n$ such
that $\overline{x}^kR=\overline{x}^{k+1}R=\ldots$. Then there
exists $r\in R$ such that $x^k-x^{k+1}r\in I$. Since $R/I$ is
flat, it follows from \cite[Theorem 4.23]{Lam} that there is an
element $a\in I$ such that $x^k-x^{k+1}r=a(x^k-x^{k+1}r)$ and
hence $x^k-x^{k+1}r=a^{k+1}(x^k-x^{k+1}r)$. Also, we have that
$I^{k+1}\subseteq (xR)^{k+1}\subseteq x^{k+1}R$ because $R$ is
right duo. Then $a^{k+1}\in x^{k+1}R$ which means that $x^k\in
x^{k+1}R$. So, $R$ is strongly $\pi$-regular. Then, by
\cite[Theorem 3]{Azum}, we may assume that $x^k=x^{k+1}r$ and
$xr=rx$ for some $r\in R$. It follows that $(x^{k-1}-x^kr)^2=0$.
But since $J=0$, $R$ is semiprime and so $x^{k-1}-x^kr=0$. If we
continue this process, then we get that $x=x^2r$. Thus, $R$ is
strongly regular.
\end{proof}

\cite[Example 4.3]{zhou} shows that Theorem~\ref{td} need not be true if
$R$ is not right duo.

\vspace{0.3in}

We obtain some conditions under which a right $G$-$\delta$-perfect ring
is $\delta$-semiregular by Theorems~\ref{r} and ~\ref{td}.
\begin{cor} Assume that $R$ is a right duo ring or a ring
such that every cyclic flat right $R$-module is projective.
If $R$ is right $G$-$\delta$-perfect, then $R$ is $\delta$-semiregular.
\end{cor}

Recall that a ring $R$ is said to be {\em right max} if every
nonzero right $R$-module has a maximal submodule. Due to Hamsher
\cite{Ham}, if $R$ is commutative, then $R$ is right max if and
only if $R/J(R)$ is regular and $J(R)$ is right $T$-nilpotent. By Theorems~\ref{t3} and ~\ref{td} we
have the following corollaries  as generalizations of
\cite[Corollaries 2.9 and 2.10]{AAES}.

\begin{cor}If $R$ is a commutative $G$-$\delta$-perfect ring, then
$R$ is a max ring. In particular, every prime ideal of $R$ is
maximal.
\end{cor}

\begin{cor} Let $R$ be a commutative $G$-$\delta$-perfect ring.
Then a module $M$ is Noetherian if and only if $M$ is Artinian.
\end{cor}

We obtain the following result by a proof similar to that of \cite[Theorem 3.3]{AAES}. We give the proof for
completeness' sake.

\begin{thm}\label{tf} Let $R$ be a ring. If $R/\delta_r$ is regular and $J(R/S_r)$ is right $T$-nilpotent, then
every module of the form $F/K$, where $F$ is a free module and $K$ is a countably generated submodule
of $F$, has a flat $\delta$-cover.
\end{thm}

\begin{proof} Let $K=\sum_{i=1}^{\infty} x_iR$ and $K_n=\sum_{i=1}^n x_iR$ for each $n\geq 1$.
Then $K=\underset{\rightarrow} \lim K_n$ and $F/K=\underset{\rightarrow} \lim F/K_n$. By \cite[Theorem 4.26(c)]{Lam},
$F/K_n$ is the direct sum of a finitely presented and a free module. It follows from \cite[Theorem 3.6]{zhou} that
$F/K_n$ has a projective $\delta$-cover. Let $\phi_n: P_n\rightarrow F/K_n$ be the epimorphism with $P_n$ projective
and $L_n=Ker(\phi_n)\ll_{\delta} P_n$. So $L_n\subseteq P_n\delta_r$. Let $\pi_n: F/K_n\rightarrow F/K_{n+1}$ be the
natural epimorphism. As $P_n$ is projective there is a homomorphism $\alpha_n: P_n\rightarrow P_{n+1}$ such that
$\phi_{n+1}\alpha_n=\pi_n\phi_n$. We also have that $\alpha_n(L_n)\subseteq L_{n+1}$. Hence, we obtain that
$0\rightarrow L_n\rightarrow P_n\rightarrow F/K_n\rightarrow 0$ is a directed system of exact sequences. Let
$L=\underset{\rightarrow} \lim L_n$, $P=\underset{\rightarrow} \lim P_n$, $\beta_n:L_n\rightarrow L$
and $\gamma_n:P_n\rightarrow P$. So we obtain the exact sequence
 $0\rightarrow L\stackrel{i} \rightarrow P\stackrel{\phi} \rightarrow F/K\rightarrow 0$.
For any $x\in L$ there is an integer $n\geq 1$ such that
$x=\beta_n(y)$ for some $y\in L_n\subseteq P_n\delta_r$. Thus,
$i(x)=i(\beta_n(y))=\gamma_n(y)\in
\gamma_n(P_n\delta_r)=\gamma_n(\delta(P_n))$.
 By Theorem~\ref{t1}, $\delta(P_n)\ll_{\delta} P_n$.
So it follows from \cite[Lemma 1.3(2)]{zhou} that $\gamma_n(\delta(P_n))\ll_{\delta} P$.
Hence, $P=\underset{\rightarrow} \lim P_n$
is a flat module and $Ker(\phi)=i(L)\ll_{\delta} P$. Thus, $\phi:P\rightarrow F/K$ is a flat $\delta$-cover of $F/K$.
\end{proof}

\begin{cor} Let $R$ be a right max ring with $R/\delta_r$ regular.
Let $F$ be a free module and $K\subseteq F$. Suppose that
$\Omega=\{ T\subseteq F|$ $T$ is an essential maximal submodule of
$F$ not containing $K\}$ is countable. Then $F/K$ has a flat
$\delta$-cover.
\end{cor}

\begin{proof} If $\Omega=\emptyset$, then $K\subseteq T$  for every essential maximal submodule $T$ of $F$.
Then $K\subseteq \delta(F)$ and so $K\ll_{\delta} F$ because
$\delta(F)\ll_{\delta} F$ by Theorem~\ref{t1}. Hence, the natural
epimorphism $F\rightarrow F/K$ is a flat $\delta$-cover of $F/K$.

Suppose that $\Omega\neq\emptyset$. For each $T\in \Omega$, let
$x_T\in K\setminus T$. Consider $L=\underset{T\in \Omega} \sum
x_TR\subseteq K$. So by Theorem~\ref{tf}, $F/L$ has a flat
$\delta$-cover. Now we will show that the natural epimorphism
$F/L\rightarrow F/K$ has a $\delta$-small kernel. Suppose that
$K/L+U/L=F/L$ and $F/U$ is singular, where $L\subseteq U\subseteq
F$. Since $F$ is projective, $U$ is an essential submodule of $F$.
If $F/U\neq 0$, then it has a maximal submodule $H/U$. Hence, $H$
is an essential maximal submodule of $F$ not containing $K$. But
then $x_H\in L\subseteq U\subseteq H$, a contradiction. Hence,
$F/U=0$ and so $K/L\ll_{\delta} F/L$.
\end{proof}

It is easy to observe that if $J(R/S_r)$ is right $T$-nilpotent, then $J(R)$ is right $T$-nilpotent, too.
Therefore, it follows from Theorem~\ref{t3} that if $R$ is a semilocal ring, then $R$ is right
$G$-$\delta$-perfect if and only if $R$ is right $G$-perfect. But we do not know an example of a
$G$-$\delta$-perfect ring that is not $G$-perfect.

\section{Some characterizations of $\delta$-semiperfect and $\delta$-perfect rings}

We start this section with some characterizations of
$\delta$-semiperfect rings. Firstly, we consider generalized
(locally) projective $\delta$-covers.
\begin{thm}\label{s1} Let $R$ be a ring. Suppose that idempotents lift modulo $\delta_r$.
Then the following statements are equivalent:

1) $R$ is $\delta$-semiperfect.

2) Every simple right $R$-module has a generalized locally projective $\delta$-cover.

3) Every simple right $R$-module has a generalized projective $\delta$-cover.
\end{thm}

\begin{proof} $(1)\Leftrightarrow (3)$ It follows from \cite[Lemma 4.3]{ANO} and \cite[Theorem 3.6]{zhou}.

$(1)\Rightarrow (2)$ It is obvious.

$(2)\Rightarrow (1)$ To show that $\overline{R}=R/\delta_r$ is semisimple we need to prove each simple
right $\overline{R}$-module $S$ is projective. If we regard $S$ as a simple $R$-module, then $S$ has a
generalized locally projective $\delta$-cover $f:P\rightarrow S$. Since $P$ is locally projective,
$Ker(f)\subseteq P\delta_r$ by Proposition~\ref{p1}.

If $P\delta_r=P$, then
$S=f(P)=f(P\delta_r)=f(P)\delta_r=S\delta_r=0$, which is
impossible. Then $P\delta_r\neq P$. Since $Ker(f)$ is maximal in
$P$, we have that $Ker(f)=P\delta_r$ and so $P/P\delta_r\cong S$.
Since $P$ is a locally projective $R$-module, $P/P\delta_r$ is a
locally projective $\overline{R}$-module and so $S$ is a locally
projective $\overline{R}$-module. But $S$ is simple so that it is
projective. Thus, $\overline{R}$ is semisimple.
\end{proof}

\begin{cor} Let $R$ be a ring. Suppose that idempotents lift modulo $\delta_r$.
Then the following statements are equivalent:

1) $R$ is $\delta$-semiperfect.

2) Every finitely generated (cyclic) right $R$-module has a generalized locally projective $\delta$-cover.

3) Every finitely generated (cyclic) right $R$-module has a generalized projective $\delta$-cover.
\end{cor}

Recall from \cite{CD} that an $R$-module $M$ is called {\em
finitely projective} if, for any finitely generated submodule
$M_0$ of $M$, there exist a finitely generated free module $F$ and
homomorphisms $f:M_0\rightarrow F$ and $g:F\rightarrow M$ such
that $g(f(x))=x$ for all $x\in M_0$. Note that a finitely
generated finitely projective module is projective. Also, it is
well-known that the following implications hold for a module:
\begin{center}
locally projective $\Rightarrow$ finitely projective $\Rightarrow$
flat.
\end{center}
Note that we will call a $\delta$- cover $f:P\rightarrow M$ of a
module $M$ a {\em locally (finitely) projective $\delta$-cover} in
case $P$ is a locally (finitely) projective module.

\begin{thm}\label{ts} The following statements are equivalent for a ring $R$:

1) $R$ is $\delta$-semiperfect.

2) Every simple right $R$-module has a locally projective $\delta$-cover.

3) Every simple right $R$-module has a finitely projective
$\delta$-cover.

4) $R/\delta_r$ is semisimple and every simple right $R$-module
has a flat $\delta$-cover.
\end{thm}

\begin{proof} The implications $(1)\Rightarrow (2)\Rightarrow (3)$ and
$(1)\Rightarrow (4)$ are obvious.

$(3)\Rightarrow (1)$ Let $S$ be a simple right $R$-module and
$f:P\rightarrow S$ be a finitely projective $\delta$-cover. Since
$S$ is cyclic, by Lemma~\ref{b5}, there exists a cyclic direct
summand $P'$ of $P$ such that $f|_{P'}$ is a finitely projective
$\delta$-cover of $S$. Then $P'$ is projective. Hence, $S$ has a
projective $\delta$-cover. Thus, $R$ is $\delta$-semiperfect.

$(4)\Rightarrow (1)$ By \cite[Theorem 1.8]{zhou}, we can
consider $R/\delta_r=\oplus_{i=1}^n S_i$, where $S_i$ is simple
singular $R$-module for each $i=1,\ldots,n$. It is enough to show
that each simple singular $R$-module has a projective
$\delta$-cover in order to prove that $R$ is $\delta$-semiperfect.
Let $M$ be a simple singular $R$-module. Then $M$ is isomorphic to
one of $S_i$'s. By hypothesis, each $S_i$ has a flat
$\delta$-cover. Let $f_i:F_i\rightarrow S_i$ be the flat
$\delta$-cover of $S_i$. Then $f=\oplus_{i=1}^n
f_i:F=\oplus_{i=1}^n F_i \rightarrow R/\delta_r$ is a flat
$\delta$-cover of $R/\delta_r$ by Lemma~\ref{b1}. Since $R$ is
projective, there exists $g:R\rightarrow F$ such that the below
diagram is commutative, where $\pi:R\rightarrow R/\delta_r$ is the
natural epimorphism.

\[\begin{diagram}
\node[2]{R}\arrow{sw,t,..}{g}\arrow{s,r} {\pi} \\
\node{F}\arrow{e,t}{f}\node{R/\delta_r}\arrow{e}\node{0}
\end{diagram}\]

Then $F=Ker(f)+Im(g)$. But $Ker(f)\ll_{\delta} F$ so that
$F=Y\oplus Im(g)$, where $Y$ is a projective semisimple submodule
of $Ker(f)$. Since $Ker(g)\subseteq Ker(\pi)=\delta_r$,
$g:R\rightarrow Im(g)$ is a projective $\delta$-cover of $Im(g)$.
Hence, $g\oplus id_Y:P=R\oplus Y\rightarrow F$ is a projective
$\delta$-cover of $F$. By Lemma~\ref{b5}, we can assume that each
$F_i$ is cyclic. It follows from Proposition~\ref{p4} that each
$F_i$ has a projective $\delta$-cover. Thus, $S_i\cong M$ has a
projective $\delta$-cover.
\end{proof}

Recall that an $M$-projective module $M$ is {\em quasi-projective}
(see \cite{AF}) and that a module $M$ is called {\em
direct-projective} if, for every direct summand $X$ of $M$, every
epimorphism $M\rightarrow X$ splits (see \cite{H}). Note that a
quasi-projective module is direct-projective.

We need the following lemma in
order to prove the next result.

\begin{lem}\label{h} \cite{H} Let $P$ be projective and $P\oplus M$ direct projective. If there is an
epimorphism $f:P\rightarrow M$, then $M$ is projective.
\end{lem}

A $\delta$- cover $f:P\rightarrow M$ of a module $M$ is said to be
a {\em quasi-projective  (direct-projective) $\delta$-cover} in
case $P$ is a quasi-projective (direct-projective) module.

\begin{thm}\label{dp} If every right $R$-module has a direct-projective $\delta$-cover,
then every right $R$-module has a projective $\delta$-cover.
\end{thm}

\begin{proof} Let $M$ be a module and consider the epimorphism $f:F\rightarrow M$,
where $F$ is free. By assumption, $F\oplus M$ has a direct-projective $\delta$-cover.
Let $g:P\rightarrow F\oplus M$ be the direct-projective $\delta$-cover of $F\oplus M$.
Consider the canonical projection $\pi:F\oplus M\rightarrow F$. Since $F$ is projective,
we have a monomorphism $h:F\rightarrow P$ which makes the following diagram commutative.

\[\begin{diagram}
\node[2]{F}\arrow{sw,t,..}{h}\arrow{s,r} {id_F} \\
\node{P}\arrow{e,t}{\pi g}\node{F}\arrow{e}\node{0}
\end{diagram}\]

So $P=F\oplus T$, where $T=Ker(\pi g)$. Now we claim that
$\overline{g}=g|_T:T\rightarrow M$ is a projective $\delta$-cover
of $M$. To show that $Ker(\overline{g})\ll_{\delta} T$, let
$T=Ker(\overline{g})+N$, where $N\leq T$. Since $P=F\oplus
T=F+Ker(\overline{g})+N$ and $Ker(\overline{g})\subseteq
Ker(g)\ll_{\delta} P$, $P=F\oplus N\oplus Y$ for a projective
semisimple submodule $Y$ of $Ker(\overline{g})$. The equality
$P=F\oplus T=F\oplus N\oplus Y$ gives that $T=N\oplus Y$ and so by
Lemma~\ref{lz}, $Ker(\overline{g})\ll_{\delta} T$.

Now we will show that $T$ is projective. Again by projectivity of
$F$ we have the following commutative diagram.

\[\begin{diagram}
\node[2]{F}\arrow{sw,t,..}{f'}\arrow{s,r} f \\
\node{T}\arrow{e,t}{\overline{g}}\node{M}\arrow{e}\node{0}
\end{diagram}\]

Hence, $T=Ker(\overline {g})+ Im(f')$. Since $Ker(\overline{g})\ll_{\delta} T$, there exists a projective semisimple
submodule $Y$ of $Ker(\overline{g})$ such that $T=Y\oplus Im(f')$. Since
$F\oplus Im(f')\leq^{\oplus} P$, it is direct-projective. Then $Im(f')$ is projective by Lemma~\ref{h}.
It follows that $T$ is projective.
 \end{proof}

By a proof similar to that of Theorem~\ref{dp}, we can observe
that if every finitely generated right $R$-module has a direct
projective $\delta$-cover, then every finitely generated right
$R$-module has a projective $\delta$-cover. The next result is an
immediate consequence of this fact and Theorem~\ref{ts}.

\begin{cor} The following statements are equivalent for a ring $R$:

1) $R$ is $\delta$-semiperfect.

2) Every finitely generated right $R$-module has a quasi-projective
$\delta$-cover.

3) Every finitely generated right $R$-module has a direct-projective
$\delta$-cover.

4) Every finitely generated (cyclic) right $R$-module has a
locally projective $\delta$-cover.

5) Every finitely generated (cyclic) right $R$-module has a
finitely projective $\delta$-cover.

6) $R/\delta_r$ is semisimple and every finitely generated
(cyclic) right $R$-module has a flat $\delta$-cover.
\end{cor}

Next, we will deal with $\delta$-perfect rings.

\begin{thm} Let $R$ be a ring such that $J(R/S_r)$ is right $T$-nilpotent.
Then the following statements are equivalent:

1) $R$ is right $\delta$-perfect.

2) Every semisimple right $R$-module has a generalized locally projective
$\delta$-cover.

3) Every semisimple right $R$-module has a generalized projective
$\delta$-cover.
\end{thm}

\begin{proof} The equivalency $(1)\Leftrightarrow (3)$ follows from \cite[Lemma 4.3]{ANO}, \cite[Theorem 3.6]{zhou}
and \cite[Lemma 1.3]{YZ},
and the proof of $(1)\Leftrightarrow (2)$ is similar to that of
$(1)\Leftrightarrow (2)$ in Theorem~\ref{s1}.
\end{proof}

We conclude this section with the following theorem which states
some equivalent conditions for a ring to be $\delta$-perfect.
\begin{thm}\label{tp} The following statements are equivalent for a ring $R$:

1) $R$ is right $\delta$-perfect.

2) Every right $R$-module has a quasi-projective $\delta$-cover.

3) Every right $R$-module has a direct-projective $\delta$-cover.

4) Every semisimple right $R$-module has a locally projective
$\delta$-cover.

5) Every semisimple right $R$-module has a finitely projective
$\delta$-cover.

6) $R/\delta_r$ is semisimple and every semisimple right
$R$-module has a flat $\delta$-cover.
\end{thm}

\begin{proof} The implications $(1)\Rightarrow (2)\Rightarrow (3)$ and
$(1)\Rightarrow (4)\Rightarrow (5)$ are obvious.

$(3)\Rightarrow (1)$ It follows from Theorem~\ref{dp}.

$(5)\Rightarrow (6)$ By Theorem~\ref{ts}, $R$ is
$\delta$-semiperfect. Every semisimple right $R$-module has a flat
$\delta$-cover since finitely projective modules are flat.

$(6)\Rightarrow (1)$ By \cite[Theorem 3.8]{zhou}, we only need to
prove that $J(R/S_r)$ is right $T$-nilpotent. Since $R/\delta_r$
is semisimple, $F/\delta(F)$ is a semisimple $R$-module for a
countably generated free module $F$ and so $F/\delta(F)$ has a
flat $\delta$-cover by assumption. Hence, the rest of the proof is
similar to that of Theorem~\ref{t3}.
\end{proof}

The condition that $R/\delta_r$ is semisimple in Theorems~\ref{ts}
and~\ref{tp} is not superfluous because of Example~\ref{e1}.

\vspace{0.3in}

\noindent {\it Acknowledgments}. This work was supported by
The Scientific Technological Research Council of
Turkey (T\"UB\.ITAK). The author would like to thank her supervisor
Prof. A. \c {C}i\u {g}dem \"{O}zcan for her advice and support throughout. This work was completed during the author's visit to Center of
Ring Theory and its Applications, Ohio University in $2010$.

\end{document}